\documentclass{article}
\usepackage{amssymb,latexsym,amsthm}
\usepackage{amsmath,amsfonts}
\usepackage{enumerate}
\usepackage{geometry}
\usepackage{url}
\ifdefined\directlua
\usepackage{fontspec}
\usepackage{polyglossia}
\setmainlanguage{english}
\usepackage{microtype}
\usepackage{pdftexcmds}
\makeatletter
\ifcase\pdf@shellescape
\relax\or
\begingroup
  \catcode`\%=12\relax
  \gdef\sformat{"Date: 
\endgroup
\directlua{
 local cmd="git show -s --format='"..\sformat.."'"
 local r=io.popen(cmd):read("*a")
 if (r) then
      tex.print("\string\\def\string\\COMMIT{"..r.."}")
 end
 }
\or
\relax\fi
\makeatother
\ifdefined\COMMIT
        \usepackage{background}
        \backgroundsetup{%
         pages=all, placement=bottom,angle=0,scale=2,%
         vshift=20pt,%
         contents={\COMMIT}}
\fi
\else
\usepackage[british]{babel}
\fi

\newcommand{\KK}{\mathbb{K}}
\newcommand{\cS}{\mathcal{S}}
\newcommand{\cQ}{\mathcal{Q}}
\newcommand{\cB}{\mathcal B}
\newcommand{\cU}{\mathcal U}
\newtheorem*{mth}{Main Theorem}
\newcommand{\wt}{\mathrm{wt}}
\newcommand{\cL}{\mathcal{L}}
\newcommand{\GG}{\mathbb{G}}
\newcommand{\fQ}{\mathfrak{Q}}
\newcommand{\fP}{\mathfrak{P}}
\newcommand{\fA}{\mathfrak{A}}
\newcommand{\fC}{\mathfrak{C}}
\newcommand{\cM}{\mathcal{M}}
\newcommand{\cC}{\mathcal{C}}
\newcommand{\cV}{\mathcal{V}}
\newcommand{\cW}{\mathcal{W}}
\newcommand{\cT}{\mathcal{T}}
\newcommand{\cX}{\mathfrak{X}}
\newcommand{\fB}{\mathfrak{B}}
\newcommand{\cN}{\mathcal{N}}
\newcommand{\LL}{\mathbb{L}}
\newcommand{\TT}{\mathbb{T}}
\newcommand{\fX}{\mathfrak{X}}
\newcommand{\cG}{\mathcal{G}}
\newcommand{\ccQ}{\mathcal{Q}}
\newcommand{\RM}{\mathrm{RM}\,}
\newcommand{\trace}{\mbox{\itshape trace}}
\newcommand{\diag}{\mbox{\itshape diag}}
\newcommand{\spin}{\text{\itshape spin}}
\newcommand{\bE}{\mathbb E}
\newcommand{\baM}{\overline{M}}
\newcommand{\PG}{\mathrm{PG}}
\newcommand{\Sp}{\mathrm{Sp}}
\newcommand{\GL}{\mathrm{GL}}
\newcommand{\F}{\mathbb{F}}
\newcommand{\FF}{\mathbb{F}}
\newcommand{\NN}{\mathbb{N}}
\newcommand{\Rad}{\mathrm{Rad}(\varphi)}
\newcommand{\Res}{\mathrm{Res}}
\newcommand{\rank}{\mathrm{rank}\,}
\newcommand{\N}{\mathcal{N}}
\newcommand{\cP}{\mathcal{P}}
\newcommand{\ox}{\overline{x}}
\newcommand{\ov}{\overline{v}}
\newcommand{\oy}{\overline{y}}
\newcommand{\oU}{\widetilde{U}}
\newcommand{\oS}{\overline{S}}
\newcommand{\oM}{\overline{M}}
\newcommand{\ou}{\overline{u}}
\newcommand{\oV}{\overline{V}}
\newcommand{\oPi}{{\overline{\Pi}}_{\varphi}}
\newcommand{\oRad}{\overline{\mathrm{Rad}}(\varphi)}

\newcommand{\R}{\mathrm{Rad}(\varphi_{\Pi})}

\newcommand{\bZ}{\bf{0}}
\newcommand{\codim}{\mathrm{codim}\,}
\theoremstyle{plain}
\newtheorem{lemma}{Lemma}[section]
\newtheorem{theorem}[lemma]{Theorem}
\newtheorem{corollary}[lemma]{Corollary}

\newtheorem{claim}{Claim}
\theoremstyle{definition}
\newtheorem{remark}[lemma]{Remark}
\newcommand\sqval{\frac{q-1}{2}\left(q^{2n-(r+d)}-q^{2n-(r+d)-1}+q^{n-(r+d)/2}+
q^{n-(r+d)/2-1}\right)}
\def\nonsquare{\ensuremath{%
    \setbox0\hbox{$\square$}%
    \rlap{\hbox to \wd0{\hss\slash\hss}}\box0
}}
\begin{document}
\title{Minimum distance of Orthogonal Line-Grassmann Codes in even characteristic}
\author{Ilaria Cardinali and Luca Giuzzi}
\date{}
\maketitle
\begin{abstract}
  In this paper we determine the minimum distance  of orthogonal
  line-Grassmann codes for $q$ even. The case $q$ odd was solved in 
  \cite{ILP}.
  For $n\neq 3$ we also determine the second smallest distance.
  Furthermore, we show that for $q$ even
  all minimum weight codewords are equivalent and that
  symplectic line-Grassmann codes are proper subcodes of codimension $2n$
of the orthogonal ones.
\end{abstract}

\section{Introduction}\label{Introduction}
A projective code $\cC(\Omega)$ is an error correcting code determined by a projective system, that is a set $\Omega$
of $N$ distinct points of
a finite projective space.  More in detail, $\cC(\Omega)$ is a
linear code of length $N$ generated by the rows of a matrix $G$ whose columns are the coordinates
of the points of $\Omega$ with respect to some fixed reference
system. In general, $\cC(\Omega)$ is not uniquely determined by $\Omega$, but it turns out to
be unique up to monomial equivalence; as such its metric properties
with respect to Hamming's distance depend only on the set of points under consideration.
With a slight abuse of notation, which is however customary when dealing
with projective codes, we shall speak of $\cC(\Omega)$ as \emph{the} code defined
by $\Omega=\{ \omega_1,\omega_2,\dots, \omega_N\}$ where the $\omega_i$s  are fixed vector representations of the points of the projective system.

As mentioned above, the parameters $[N,K,d_{\min}]$ of $\cC(\Omega)$ depend only on  the pointset
$\Omega$; in particular, the length $N$ is the size of $\Omega$ and the dimension $K$
is the (vector) dimension of the subspace of $W$ spanned
by $\Omega$. It is straightforward to show that the minimum distance $d_{\min}$ is
\begin{equation}\label{min distance 1}
 d_{\min}=N-\max_{\Pi}|\Omega\cap\Pi|,
\end{equation}
as $\Pi$ ranges among all hyperplanes of the space $\PG(\langle\Omega\rangle)$;
 we refer to \cite{TVZ}
for further details.

The codes associated with polar $k$-Grassmannians of either orthogonal, symplectic or Hermitian
type have been  introduced  respectively in \cite{IL0}, \cite{IL2} and
\cite{IL3}.
In the case of line-Grassmannians, that is for $k=2,$ the following results are known:
in the symplectic case it has been shown in~\cite{IL2} that the minimum distance is
 $q^{4n-5}-q^{2n-3}$ for any $q$; in the orthogonal case it has been shown in \cite[Main Result 2]{IL0}
that the minimum distance is $d_{\min}=q^{3}-q^{2}$ for $n=2$ for any $q$ and in
\cite{ILP} that the minimum distance is $d_{\min}=q^{4n-5}-q^{3n-4}$ for $q$ odd.
Similar results hold in the Hermitian case, see \cite{IL3}.

The aim of the present paper is to determine the minimum distance of orthogonal
line-Grassmann codes for $q$ even. We are also able to determine the second
smallest distance in almost all cases.
Our  main result is the following.
\begin{mth}\label{main1}
For $q$ even,
  the minimum distance of a line orthogonal Grassmann code
  is
  \[ d_{\min}=q^{4n-5}-q^{3n-4} \]
  and  all words of minimum weight are projectively
  equivalent.
  Furthermore, the second smallest distance for $n\neq3$ is
  $q^{4n-5}-q^{2n-3}$.
\end{mth}

Using the aforementioned results of \cite{IL0,ILP} this leads to the following general result.

\begin{corollary}
  The parameters $[N,K,d_{\min}]$  of a line orthogonal Grassmann code are
 \[ N=\frac{(q^{2n}-1)(q^{2n-2}-1)}{(q-1)(q^2-1)},\quad K=\begin{cases} (2n+1)n & \mbox{$q$ odd} \\
(2n+1)n-1 & \mbox{$q$ even},\end{cases} \quad d_{\min}=q^{4n-5}-q^{3n-4}. \]
\end{corollary}
Note that for $q$ odd and $n=2$, by \cite[Corollary 3.8]{ILP}, the minimum weight codewords lie on two orbits under the action of the linear automorphism group of the code.
\bigskip

The structure of the paper is as follows.
In Section~\ref{PRE} we set the notation and introduce some preliminary results. In particular,
in Section~\ref{pre 1} we recall some basic results on polar Grassmannians and their associated codes and in Section~\ref{pre 2} we describe in detail a fundamental formula for the computation of weights of codewords in a projective code.
In Section~\ref{sec3} we shall prove our Main Theorem.

For further details on the actual construction of orthogonal and symplectic line-Grassmann codes, we refer to~\cite{IL1}  where some efficient algorithms for encoding, decoding and error-correction have
been presented.

\section{Preliminaries}
\label{PRE}
\subsection{Grassmann and Polar Grassmann codes}\label{pre 1}
Let $V:=V(2n+1,q)$ be a vector space of odd dimension defined over  a finite field $\FF_q$ of order $q$ and denote by
$\cG_{k}$ the Grassmannian of the $k$--subspaces of $V.$

For any $k<\dim(V)$, let $\varepsilon_k:\cG_{k}\to \PG(\bigwedge^k V)$ be
the usual Pl\"ucker embedding,
 mapping a point $\langle v_1,\ldots,v_k\rangle$ of $\cG_k$ to the projective point $\langle v_1\wedge\cdots\wedge v_k\rangle$ of $\PG(\bigwedge^k V)$:
\[
\varepsilon_k:\langle v_1,\ldots,v_k\rangle\mapsto \langle v_1\wedge\cdots\wedge v_k\rangle.
\]

Let $\eta:V\to\FF_q$ be a fixed non-degenerate quadratic form over $V$ and
denote by $\Delta_k$ the {\it orthogonal Grassmannian} associated to $\eta$, that is $\Delta_k$ is the geometry whose points are the $k$--subspaces of $V$ which are
totally singular for $\eta$ and whose lines are defined as follows
\begin{itemize}
  \item  if $k<n$, then $\ell_{X,Y}:=\{Z\colon  X<Z<Y\colon \dim Z=k\}$, 
with $\dim X=k-1$, $\dim Y=k+1$ and $Y$ totally singular;
  \item  if  $k=n$, then $\ell_{X}:=\{Z\colon X<Z<X^{\perp_{\eta}}\colon \dim Z=n\}$, with $\dim X=n-1$, $Z$ totally singular and $X^{\perp_{\eta}}:=\{y\in V\colon \beta(x,y)=0 \textrm{  for all }x\in X\}$, where $\beta$ is the sesquilinearization of $\eta.$

\end{itemize}

For $k<n$, $\Delta_{k}$ is a proper subgeometry of $\cG_k$. In any case, for $k\leq n$ the point-set of
$\Delta_k$ is always a subset of that of $\cG_k$.

Put  $\varepsilon_k(\mathcal{G}_{k}):=\{\varepsilon_k(X)\colon X \text{ is a point of }\mathcal{G}_{k}\}$ and
$\varepsilon_k(\Delta_{k})=\{\varepsilon_k(Y)\colon Y \text{ is a point of }\Delta_{k}\}.$
Then, the above statement reads as
$\varepsilon_k(\Delta_{k})\subseteq\varepsilon_k(\cG_{k})\subseteq\PG(\bigwedge^k V)$.

We warn the reader that throughout the paper we will consider vectors and vector dimensions but we will adopt projective terminology.

\begin{theorem}[\cite{IP13}]
\label{ipt}
 Let $\varepsilon_k:\Delta_{k}\to\PG(\bigwedge^kV)$ be the
  restriction of the Pl\"ucker embedding to the orthogonal
  Grassmannian $\Delta_{k}$ and let $W_{k}:=\langle\varepsilon_k(\Delta_{k})\rangle$. Then,
           \[ \dim W_{k}=\begin{cases}
               {{2n+1}\choose{k}} & \text{ if $q$ is odd } \\
               \binom{2n+1}{k}-\binom{2n+1}{k-2} & \text{ if $q$ is even.}\\
           \end{cases}\]
  \end{theorem}
For more information on embeddings of orthogonal line-Grassmannians we refer to~\cite{IP14}.

The image $\varepsilon_k(\cG_k)$ is a projective system in $\PG(\bigwedge^kV)$ and the projective code $\cC(\cG_k)$ is called \emph{ $k$-Grassmann code}. Grassmann codes have been introduced in \cite{R1,R2} as generalizations of first order Reed--Muller codes and they have been extensively investigated ever since;
see \cite{Gl0,Go,No,R1,R2}. Their parameters, as well as some of their higher weights, have been fully
determined in \cite{No}.

The set $\Omega:=\varepsilon_k(\Delta_{k})$ is a projective system of $\PG(W_k)\subseteq\PG(\bigwedge^kV)$, hence it is natural to consider the projective code $\cP_{n,k}:=\cC(\Omega)$ arising from  $\Omega.$
The codes $\cP_{n,k}:=\cC(\Omega)$ are called \emph{orthogonal $k$-Grassmann codes} and they were introduced in \cite{IL0}.
Theorem~\ref{ipt} immediately provides the
length $N$ as the number of points of $\Delta_k$ and the dimension $K=\dim W_{k}$ of $\cP_{n,k}$.
A more difficult task is to determine the minimum distance of an orthogonal Grassmann code. In~\cite{IL0} we obtained the exact value of $d_{\min}$ for $n=k=2$ and $n=k=3$; more
recently, in \cite{ILP}, it has been shown that for $q$ odd and $k=2$ the minimum distance of
$\cP_{n,2}$ is $q^{4n-5}-q^{3n-4}$.
\par
We now present in detail a geometric setting in which it is possible to study the weights of a
projective code arising from the image under the Pl\"ucker embedding $\varepsilon_k$ of an arbitrary set of $k$-subspaces.

For any vector space $U$, denote by $U^*$ its dual.
It is well known that
$(\bigwedge^kV)^*\cong\bigwedge^kV^*$. Suppose
$\Omega=\{\omega_1,\ldots,\omega_N\}\subseteq\varepsilon_k(\cG_k)$
to be a projective system of $\bigwedge^kV$ and take $W:=\langle\Omega\rangle$.
Let now
\[ \cN(\Omega):=\{ \varphi\in\bigwedge^kV^*\colon \varphi|_{\Omega}\equiv 0 \} \]
be the annihilator of the set $\Omega$; clearly $\cN(\Omega)=\cN(W)$.
There exists a correspondence between the elements of
 $(\bigwedge^kV^*)/\cN(\Omega)\cong W^*$ and the codewords of $\cC(\Omega)$.
More precisely, given any $\varphi\in W^*,$
the codeword $c_{\varphi}$ corresponding to $\varphi$ is defined as
\[ c_{\varphi}:=(\varphi(\omega_1),\ldots,\varphi(\omega_N)). \]
As $\Omega$ spans $W$ it is immediate to see that $c_{\varphi}=c_{\psi}$ if and only if
$\varphi-\psi\in\cN(\Omega)$, that is $\varphi=\psi$ as elements of $W^*$.

Define the weight $\wt(\varphi)$ of $\varphi$ to be the weight of $c_{\varphi}$, that is
\[ \wt(\varphi):=\wt(c_{\varphi})=|\{ \omega\in\Omega \colon \varphi(\omega)\neq 0 \}|. \]

It is well known that linear functionals in $\bigwedge^kV^*$ are equivalent
to $k$-linear alternating forms defined on $V$.
In particular, given $\varphi\in\bigwedge^kV^*$ we can define $\varphi^*\colon V^k\to\FF_q$ as
\[ \varphi^*(v_1,\ldots,v_k):=\varphi(v_1\wedge v_2\wedge\cdots\wedge v_k) \]
which is a $k$-linear alternating form. Conversely, given a $k$-linear alternating
form $\varphi^*\colon V^k\to\FF_q$, there is a unique element $\varphi\in\bigwedge^kV^*$
such that
\[ \varphi(v_1\wedge\ldots\wedge v_k):=\varphi^*(v_1,\ldots,v_k) \]
for any $v_1,\ldots,v_k\in V$.
In particular, $\varphi(u)=0$ for $u=\langle v_1\wedge v_2\wedge\cdots\wedge v_k\rangle\in\Omega$ if and only if
all the $k$-tuples of elements of the vector space $U:=\langle v_1,\ldots,v_k\rangle$ are killed by
$\varphi^*$.
With a slight abuse of notation, in the remainder of this paper
we shall use the same symbol $\varphi$ for both the linear functional and the
related $k$-alternating form.


For linear codes the minimum distance is the minimum of the weights of the non-zero codewords;
so, in order to obtain the minimum distance of the codes $\cP_{n,k}$ we need to determine the
maximum number of $k$--spaces of $V$ with are both totally $\eta$--singular
 and $\varphi$--totally isotropic,
where  $\varphi$ is an arbitrary $k$--linear alternating form which is not identically null
on the elements of $\Delta_k$.

\subsection{A recursive formula for the weights}\label{pre 2}
Take $\varphi\in(\bigwedge^kV)^*$ and let $u\in V$.
Then, we can define a
functional $\varphi_u\in(\bigwedge^{k-1}V)^*$ by
\[ \varphi_u:\begin{cases}
\bigwedge^{k-1}V \rightarrow \F_q\\
x\mapsto \varphi (u\wedge x).
  \end{cases} \]
Define $u\bigwedge^{k-2}V:=\{u\wedge y: y\in\bigwedge^{k-2} V\}\subseteqq \bigwedge^{k-1} V$ and $V_u:=V/\langle u\rangle$.
Observe that for any $y\in u\bigwedge^{k-2}V$ we have $\varphi_u(y)=0$.
Since
$\bigwedge^{k-1}V_u= (\bigwedge^{k-1}V)/(u\bigwedge^{k-2}V)$, for $x\in \bigwedge^{k-1}V$, we can define a functional $\widehat{\varphi}_u\in (\bigwedge^{k-1}V_u)^*$ by
\[ \widehat{\varphi}_u:\begin{cases}
    \bigwedge^{k-1}V_u \to \FF_q \\
    x+(u\bigwedge^{k-2}V)\mapsto\varphi_u(x).
  \end{cases} \]

Let now $\cT$ be a set of $k$-subspaces of $V$ and denote by $\TT\subseteq V$ the set of vectors of $V$ belonging to at least one element of $\cT.$ Put $\varepsilon_k(\cT)=\{\varepsilon_k(X)\colon X\in \cT\}.$ Define
\[ \cT_u:=\{ X/\langle u\rangle: X\in\cT, u\in X\}\subseteq V_u\,\,{\textrm{and} }\,\,S_u:=\langle\cT_u\rangle.\]
Apply the Pl\"ucker embedding $\varepsilon_{k-1}:\cG_{k-1}\to\PG(\bigwedge^{k-1}V_u)$
to the elements $X/\langle u\rangle\in \cT_u$. Put $\Omega_u:=\{\varepsilon_{k-1}(X/\langle u\rangle)\colon X/\langle u\rangle \in \cT_u\}.$ 
The set $\Omega_u$ can be regarded as a projective system of $\bigwedge^{k-1}S_u.$
We can consider the linear functional
 $\widetilde{\varphi}_u:=\widehat{\varphi}_u|_{\bigwedge^{k-1}S_u}$ which is the restriction of $\widehat{\varphi}_u$ to $\bigwedge^{k-1}S_u.$
As recalled in Section \ref{pre 1}, to each codeword
of $\cC(\Omega_u)$ there correspond exactly one functional
$\widetilde{\varphi}_u'\in(\bigwedge^{k-1}S_u^*)/\widetilde{\cN}(\Omega_u)$ where
\[ \widetilde{\cN}(\Omega_u):=\{ \psi\in\bigwedge^{k-1}S_u^*\colon \psi|_{\Omega_u}\equiv 0 \}. \]
Observe that, given $\widetilde{\varphi}_u'\in(\bigwedge^{k-1}S_u^*)/\widetilde{\cN}(\Omega_u)$, there exists a functional $\widetilde {\varphi}_u\in(\bigwedge^{k-1}S_u^*)$ such that  $\wt(\widetilde{\varphi}_u)=\wt(\widetilde{\varphi}_u').$

Under the set-up introduced above, the following formula holds (see \cite[Lemma 2.2]{ILP}):
\begin{equation}
\label{ka}
 \wt(\varphi)=  \frac{1}{q^k-1}\sum_{\begin{subarray}{c}
    u\in \TT\\
    \end{subarray}}\wt(\widetilde{\varphi}_u).
\end{equation}

Note that when $\cT$ is the  set of all $k$-subspaces of $V$, namely $\cT$ is $\cG_k$, then $\TT$ is the pointset of $V,$ hence $S_u=V_u,$  $\bigwedge^{k-1}(S_u/\langle u\rangle)=\bigwedge^{k-1}V_u$  and  $\widetilde{\varphi}_u=\widehat{\varphi}_u=\widetilde{\varphi}_u'.$

Observe also that if $\Omega_u$ spans $\bigwedge^{k-1}S_u$, then $\widetilde{\cN}(\Omega_u)$ is trivial and $\widetilde{\varphi}_u'=\widetilde{\varphi}_u.$
This happens, for example, in the case of orthogonal Grassmann codes for $q$ odd or $k=2$. Indeed, if we specialize to the case of line orthogonal Grassmann codes (i.e. $k=2$ and $\cT=\Delta_2$), we have that  $\TT$ is the pointset of the non-degenerate parabolic quadric $\cQ\cong Q(2n,q)$ defined by the quadratic form $\eta$. So
 $\cT_u$, with $u\in \cQ$, is isomorphic to the non-degenerate  parabolic quadric $\cQ_u$ having as points, the lines of $\cQ$ through $u.$ In this case $\cT_u\cong \cQ_u$ is naturally embedded (by $\varepsilon_1$) as a non-degenerate parabolic quadric $\Omega_u\cong Q(2n-2,q)$ in a $(2n-1)$-dimensional vector space $S_u$. Hence $\Omega_u$ spans $S_u$ and $\widetilde{\cN}(\Omega_u)=\{0\}$.

\section{Proof of the Main Theorem}\label{sec3}
If $q=2$ and $n=2$, \cite[Main Result 2]{IL0} shows that the minimum distance of
the code is $4=2^3-2^2$. A direct computation proves that all $45$ words of
minimum weight lie in the same orbit under the action of the automorphism group of
the code, which is isomorphic to the orthogonal linear group $GO(5,2)$ in its natural action on $\bigwedge^2V$.

Henceforth,  we shall assume $q$ to be even and $(q,n)\neq(2,2)$.
As mentioned in the Introduction, the case $q$ even and
$n=2$ is also covered by \cite[Main result 2]{IL0}.

As $\dim V$ is odd, all non--degenerate quadratic forms on $V$ are projectively
equivalent. So, for the purposes of the present paper, we can assume to have
 fixed a basis
$B:=(e_1,\ldots,e_{2n+1})$ of $V$ such that $\eta$ is
\begin{equation}
\label{e:eta}
 \eta(x):=\sum_{i=1}^n x_{2i-1}x_{2i}+x_{2n+1}^2,
\end{equation}
where $(x_i)_{i=1}^{2n+1}$ are the coordinates of a vector $x\in V$ with respect to $B.$

Let $\beta(x,y):=\eta(x+y)-\eta(x)-\eta(y)$ be the bilinear form associated
with $\eta$ by sesquilinearization. As $q$ is even, the bilinear form $\beta$ is degenerate with
$1$-dimensional radical $N=\{x\in V\colon \beta(x,y)=0 \text{ for any } y\in V\}=\langle e_{2n+1}\rangle$.

The set of totally singular vectors for $\eta$ determine a parabolic quadric $\cQ\cong Q(2n,q)$ in
$\PG(V)$.
We recall that when $p$ and $q$ are distinct points of $\cQ$, the line spanned by $p$ and $q$ is totally singular if and only if
\[ \eta(p)=\eta(q)=\beta(p,q)=0. \]
For any $p\in\cQ,$ define
\[ p^{\perp_\cQ}:=\langle u\in\cQ\colon \langle p,u\rangle\subseteq\cQ\rangle. \]
So $p^{\perp_\cQ}$ is the tangent hyperplane at $p$ to $\cQ$.
By construction, we also have $p^{\perp_\cQ}=\{ u\in V\colon \beta(p,u)=0 \}$; thus $N\subseteq p^{\perp_\cQ}$
for any $p\in\cQ$ and any line through $N$ is tangent to $\cQ$. The point $N$
is called the \emph{nucleus} of the quadric $\cQ$.

By Theorem~\ref{ipt}, $\Sigma:=\PG(\langle\varepsilon_2(\Delta_{2})\rangle)$ is a hyperplane of $\PG(\bigwedge^2 V)$; more in detail, $\Sigma$ is the kernel of the functional $\beta$
arising from the alternating bilinear form introduced above.

 If $\Pi$ is a hyperplane of $\PG(\bigwedge^2 V)$ different form $\Sigma$, then  $\Pi_{\Sigma}:=\Pi\cap \Sigma$ is a hyperplane of $\Sigma$; clearly, every hyperplane of $\Sigma$ can be obtained by intersecting $\Sigma$ with suitable hyperplanes of $\PG(\bigwedge^2 V).$

So, by Equation~\eqref{min distance 1},
  $d_{\min}=N-\max_{\Pi}| \varepsilon_2(\Delta_{2})\cap\Pi_{\Sigma}|,$ where $\Pi$ ranges among all hyperplanes of $\PG(\bigwedge^2 V)$ different from $\Sigma.$ Regarding $\Pi$ as (the kernel of) a linear functional
$\varphi_{\Pi}\in \bigwedge^2 V^*$
we see that the cardinality of  $ \varepsilon_2(\Delta_{2})\cap\Pi_{\Sigma}$ is the same as the number of lines of $V$ which are simultaneously totally singular for $\eta$ and totally isotropic for $\varphi_{\Pi},$ now
considered as a bilinear alternating form on $V\times V$.
Observe that by the correspondence between codewords of $\cP_{n,2}$ and elements of
$\bigwedge^2V^*/\cN(\varepsilon_2(\Delta_2))$ explained in Section~\ref{pre 1}, two functionals $\varphi,\vartheta\in\bigwedge^2V^*$
induce the same codeword $c\in\cP_{n,2}$ if and only if $\varphi-\vartheta=a\beta$, for some $a\in\FF_q$.

As $\dim(V)$ is odd, the bilinear form $\varphi_\Pi$ is always degenerate; denote by $\R$ its radical, i.e. $\R:=\{x\in V\colon \varphi_\Pi(x,y)=0\,\forall y\in V\}$.

\medskip
We are now ready to prove our main theorem. We  proceed in several stages.
 First of all we consider in Section~\ref{sec3.1} those hyperplanes $\Pi$ corresponding to alternating bilinear forms having radical $\R$ containing the nucleus $N$ of $\cQ.$
We prove in Lemma~\ref{lemma N in R} that the weight of these forms
is always at least $q^{4n-5}-q^{2n-3}$, thus showing that they cannot have minimum weight.
Then,
in Section~\ref{sec3.2} we deal with the class of hyperplanes corresponding to alternating bilinear forms having radical not containing the nucleus $N$ of $\cQ$. We show that a necessary condition
for the forms to correspond to minimal weight codewords is to have radical of maximum dimension (see Theorem~\ref{min weight}). Finally, we characterize the codewords of minimum weight.

\subsection{Weight of $\varphi_{\Pi}$ when $N\subseteq \R$}\label{sec3.1}
Suppose $\Pi$ to be a hyperplane of $\PG(\bigwedge^2 V)$ corresponding to a bilinear alternating form $\varphi_\Pi$ whose radical $\R$ contains $N.$
For the sake of simplicity, we will write $\varphi$ instead of $\varphi_{\Pi}.$

Denote by $\cW$ the non-degenerate symplectic polar space in $V/N$ having as points
the lines of $V$ through $N$ and as lines the planes of $V$ through $N$ containing a line of $\cQ.$
It is immediate to see that $\cW$ is defined by the non-degenerate alternating form ${\beta}^{\textrm{sp}}$ induced by $\beta$ on $V/N.$

The projection $\iota$ from $V$ to $V/N$ induces an isomorphism of polar spaces from $\cQ$ to $\cW.$
So, $\iota$ naturally induces an isomorphism $\overline{{\iota}}$ between the orthogonal line-Grassmannian $\Delta_2$
and the symplectic line-Grassmannian ${\Delta}_2^{\textrm{sp}}$ associated with ${\beta}^{\textrm{sp}}$  by  $\overline{{\iota}}(\ell):=\langle\ell,N\rangle/N$ for any $\ell\in\Delta_2$.
Denote by ${\cP}_{n,2}^{\textrm{sp}}$ the symplectic Grassmann code defined by ${\beta}^{\textrm{sp}}$ on $V/N$. (See \cite{IL2} for more information on symplectic Grassmann codes.)

\begin{lemma}\label{lemma N in R}
If $N\subseteq\Rad$ then $\wt(\varphi)\geq q^{4n-5}-q^{2n-3}.$
\end{lemma}
\begin{proof}
Since $N\subseteq\Rad$, we can consider the (possibly degenerate) symplectic polar space $\cW_{\varphi}$ of $V/N$ having as points the lines of $V$ (totally isotropic for $\varphi$) through $N$ and as lines the planes of $V$ through $N$ which are totally isotropic for $\varphi.$ Denote by ${\varphi}^{\textrm{sp}}$ the bilinear alternating form of $V/N$ defining $\cW_{\varphi}$. Clearly, ${\varphi}^{\textrm{sp}}$ is the form induced by
 $\varphi$  on $V/N.$
Since $N\subseteq\Rad$, given any line $\ell$ of $V$ with $N\notin\ell$ which is totally isotropic for $\varphi$, the plane $\langle N,\ell\rangle$ is also totally isotropic for $\varphi$. So $\overline{{\iota}}(\ell)$ is
a totally isotropic line for ${\varphi}^{\textrm{sp}}$.
Conversely, if $\langle \ell,N\rangle$ is a totally isotropic line for ${\varphi}^{\textrm{sp}}$, then
$\ell$ is totally isotropic for $\varphi$.
So, the projection $\iota\colon V\rightarrow V/N$ induces a bijection between the set of lines of $V$
which are simultaneously  totally singular for $\eta$ and totally isotropic for $\varphi$
and the set of lines of $V/N$ which are simultaneously totally isotropic for ${\beta}^{\textrm{sp}}$ and totally isotropic for ${\varphi}^{\textrm{sp}}.$
So $\wt(\varphi)=\wt({\varphi}^{\textrm{sp}})$, where $\wt(\varphi)$ is intended as the weight of
the codeword induced by $\varphi$ in the orthogonal line-Grassmann code ${\cP}_{n,2}$, while $\wt({\varphi}^{\textrm{sp}})$
corresponds to the weight of a codeword $c_{{\varphi}^{\textrm{sp}}}$ in the symplectic line-Grassmann code
${\cP}_{n,2}^{\textrm{sp}}$.
By the Main Theorem of \cite{IL2}, $\wt({\varphi}^{\textrm{sp}})\geq q^{4n-5}-q^{2n-3}$.
This completes the lemma.
\end{proof}
Observe that the isomorphism $\overline{{\iota}}$ is not an isomorphism between the codes $\cP_{n,2}$ and
${\cP}_{n,2}^{\textrm{sp}}$ induced by the projective systems arising from the embedding of the
respective polar Grassmannians; in particular, we see that
for $q$ even ${\cP}_{n,2}^{\textrm{sp}}$ is a proper subcode of $\cP_{n,2}$ with codimension $2n$.

\subsection{Weight of $\varphi_{\Pi}$ when $N\not\subseteq\R$}\label{sec3.2}
Suppose that $\Pi$ is a hyperplane of $\PG(\bigwedge^2 V)$ corresponding to a bilinear alternating form $\varphi_\Pi$ whose radical $\R$ does not contain $N.$
As in Section~\ref{sec3.1}, we shall write $\varphi$ instead of $\varphi_{\Pi}.$
It is not possible to proceed now as in Section~\ref{sec3.1} since, under the hypothesis $N\not\subseteq \Rad$, the form $\varphi$ does not induce any
symplectic polar space in $V/N$, as
 there are some  lines of $V$ through $N$ which are not totally isotropic for $\varphi$.

We will rely on Equation~\eqref{ka} adapted to the special case $\Omega=\varepsilon_2(\Delta_2)$.
For the convenience of the reader we write explicitly Equation~\eqref{ka} in this case:
for any  $\varphi\in\bigwedge^2V$,
\begin{equation}
\label{ka-1}
 \wt(\varphi)=  \frac{1}{q^2-1}\sum_{\begin{subarray}{c}
    u\in \cQ\\
    \end{subarray}}\wt(\widetilde{\varphi}_u)
\end{equation}
where $\widetilde{\varphi}_u:(u^{\perp_{\cQ}}/\langle u\rangle) \to \FF_q,\,\widetilde{\varphi}_u(x+\langle u\rangle):=\varphi(u,x)$ with $x\in u^{\perp_{\cQ}}$ and $u\in \cQ.$
Observe that
the vector space $u^{\perp_{\cQ}}/\langle u\rangle$  is naturally  endowed with the quadratic
form $\eta_u:x+\langle u\rangle\to\eta(x)$ and  $\dim u^{\perp_{\cQ}}/\langle u\rangle=2n-1.$  It is well known that
the set of all totally singular points
for $\eta_u$ is a parabolic quadric $\cQ_u\cong Q(2n-2,q)$ of rank $n-1$ in $u^{\perp_{\cQ}}/\langle u\rangle.$
In particular, the possible non-zero weights of $\widetilde{\varphi}_u$ correspond to the  non-trivial
hyperplane sections of $\cQ_u$. So, the following lemma is straightforward.
\begin{lemma}\label{pesi residuo}
  Either $\widetilde{\varphi}_u=0$ or
  $\wt(\widetilde{\varphi}_u)\in\{q^{2n-3}-q^{n-2},q^{2n-3},q^{2n-3}+q^{n-2}\}$.
\end{lemma}
Define
\[\begin{array}{l@{:=\,}l@{\qquad}l@{\,:=\,}l}
     \fA'& \{u\colon u\in\cQ\,\,\text{and}\,\, \widetilde{\varphi}_u\not=0 \}, & A'&|\fA'|\smash{;} \\
    \fB & \{u\colon u\in\fA'\,\,\text{and}\,\,  \wt(\widetilde{\varphi}_u)=q^{2n-3} \},
       & B&|\fB|\smash{;} \\
     \fC & \{u\colon u\in\fA'\,\,\text{and}\,\, \wt(\widetilde{\varphi}_u)=q^{2n-3}+q^{n-2}\}, &C&|\fC|\smash{.}
   \end{array}
\]
By definition, both $\fB$ and $\fC$ are subsets of $\fA'$ and $B, C\geq 0$.
Using Equation~\eqref{ka-1} and Lemma~\ref{pesi residuo} we can write
\begin{equation}
  \label{eq peso}
  \wt(\varphi)=\frac{q^{2n-3}-q^{n-2}}{q^2-1}A'+\frac{q^{n-2}}{q^2-1}B+
  \frac{2q^{n-2}}{q^2-1}C.
\end{equation}
For any $u\in V,$ write $u^{\perp_{\varphi}}:=\{x\in V\colon\varphi(u,x)=0\}\subseteq V.$

\begin{lemma}\label{lemma equiv}
Let $u\in \cQ.$ Then $\widetilde{\varphi}_u=0\Leftrightarrow u^{\perp_{\cQ}}\subseteq u^{\perp_{\varphi}}.$
\end{lemma}
\begin{proof}
Let $x\in u^{\perp_{\cQ}}$ and suppose $u^{\perp_{\cQ}}\subseteq u^{\perp_{\varphi}}.$ Then $\varphi(u,x)=0.$ So, by definition of ${\varphi}_u$,  ${\varphi}_u(x)=0$  for any $x\in u^{\perp_{\cQ}}$, whence $\widetilde{\varphi}_u=0$.
Conversely, if $\widetilde{\varphi}_u$ is identically zero, then
$\varphi_u(x)=\varphi(u,x)=0 \,\forall x\in u^{\perp_{\cQ}}$;
so $u^{\perp_\cQ}\subseteq u^{\perp_{\varphi}}.$
\end{proof}
Put
\[ S:=|\{u\in\cQ\colon \widetilde{\varphi}_u=0\}|=q^{2n}-1-A'. \]
By Lemma~\ref{lemma equiv},
$S=|\{u\in\cQ\colon u^{\perp_{\cQ}}\subseteq u^{\perp_{\varphi}} \}| =|\{u\in\cQ\colon u^{\perp_{\varphi}}=V \}| + |\{u\in\cQ\colon u^{\perp_{\cQ}}=u^{\perp_{\varphi}} \}|.$

If we define $A:=q^{2n-2}-1-S$, then
Equation~\eqref{eq peso} becomes
\begin{equation}
\label{eq peso final}
\wt(\varphi)=q^{4n-5}-q^{3n-4}+\frac{q^{n-2}}{q^2-1}((q^{n-1}-1)A+B+2C).
\end{equation}
Furthermore,
\begin{equation}\label{def A}
\begin{array}{ll}
A&=q^{2n-2}-1-|\{u\in\cQ\colon u^{\perp_{\cQ}}\subseteq u^{\perp_{\varphi}}\}|\\
 & =q^{2n-2}-1-| \{\Rad\cap \cQ\}|-|\{u\in\cQ\colon u^{\perp_{\cQ}}=u^{\perp_{\varphi}}\}|.
\end{array}
\end{equation}
In particular, as $B,C\geq 0$, if $A>0$, then $\wt(\varphi)>q^{4n-5}-q^{3n-4}$.

We shall first consider non-null bilinear forms $\varphi$ whose radical is not maximum and show that they
cannot give words of weight $q^{4n-5}-q^{3n-4}$. Then we shall study in detail the weights
arising from bilinear forms $\varphi$ whose radical has dimension $2n-1$.

\begin{lemma}\label{lemma N not in R}
  If $N\not\subseteq \Rad$ and $\dim(\Rad)<2n-1$, then
  $A> 0.$
  If $n\geq 3$, then $A\geq q^{2n-2}-q^{2n-3}-q^2>0$.
\end{lemma}
\begin{proof}
Write $\dim(\Rad)=2(n-r)+1$ with $1<r\leq n.$
By Equation~\eqref{def A},
in order to prove $A>0$,
we need to provide a suitable upper bound on the cardinality of the set $\{u\in\cQ\colon u^{\perp_{\cQ}}\subseteq u^{\perp_{\varphi}}\}.$

Since $N$  is the nucleus of the quadric $\cQ,$ we have $N\subseteq u^{\perp_\cQ}$ for any $u\in \cQ$; so,
 $N\subseteq u^{\perp_\varphi}$ for any $u\in\cQ$ such that  $u^{\perp_\cQ}\subseteq u^{\perp_\varphi}$.
Consequently, $u\in N^{\perp_\varphi}$ for any $u\in\cQ$ such that  $u^{\perp_\cQ}\subseteq u^{\perp_\varphi}$.
Note that $\Pi_{\varphi}:=N^{\perp_\varphi}$ is a (proper) hyperplane of $V$ since $N\not\subseteq \Rad.$
Equation~\eqref{def A} can now be rewritten as
\begin{equation}\label{def A-1}
A=q^{2n-2}-1-| \{\Rad\cap \cQ\}|-|\{u\in\cQ\cap \Pi_{\varphi}\colon u^{\perp_{\cQ}}=u^{\perp_{\varphi}}\}|.
\end{equation}

Denote by $\overline{V}$ the $2n$-dimensional vector space spanned by $\overline{B}:=(e_i)_{i=1}^{2n}$
and for any vector ${x}\in V$ of coordinates $(x_i)_{i=1}^{2n+1}$ with respect to the basis
$B$,  let $\overline{x}$ be the vector of $\overline{V}$ with coordinates
$(x_i)_{i=1}^{2n}$ with respect to  $\overline{B}.$ The map $\xi:x\to\overline{x}$
is clearly linear  from $V$ to $\oV$.
 We warn the reader that $\overline{V}$ shall not be regarded
as a subspace of $V$.
Given any $\ox=(x_i)_{i=1}^{2n}\in\oV$ there is exactly one vector $x\in\cQ$ such
that $x=(x_i)_{i=1}^{2n+1}$; here,  $x_{2n+1}=\left( x_1x_2+\cdots+x_{2n-1}x_{2n}\right)^{1/2}$.
In particular the restriction of $\xi$ to $\cQ$ is a bijection to $\oV$.

Let $S$ be the $(2n+1)\times (2n+1)$-antisymmetric matrix representing $\varphi$ with respect to $B$ and consider the alternating bilinear form $\overline{\varphi}\colon \overline{V} \times \overline{V}\rightarrow \F_q$
represented (with respect to $\overline{B}$) by the matrix $\overline{S}$ obtained from $S$ by removing its last row and its last column.
More explicitly, if
\[ S:=\begin{pmatrix}
    0 & s_{12} & \ldots & s_{1,2n} & s_{1,2n+1} \\
    s_{12} & 0 & \ldots & s_{2,2n} & s_{2,2n+1}  \\
    \vdots &  \vdots  &         & \vdots   & \vdots \\
    s_{1,2n+1}  & s_{2,2n+1}  & \ldots & s_{2n,2n+1}  & 0
    \end{pmatrix}\qquad {\rm then }\qquad
  \overline{S}:=\begin{pmatrix}
    0 & s_{12} & \ldots & s_{1,2n} \\
    s_{12} & 0 & \ldots & s_{2,2n} \\
    \vdots &  &        & \vdots \\
    s_{1,2n} & s_{2,2n} & \ldots & 0
    \end{pmatrix}.
\]
Analogously, let $M$ be the $(2n+1)\times (2n+1)$-antisymmetric matrix representing the bilinear form $\beta$ associated to the quadratic form $\eta$ (see Equation~\eqref{e:eta}) with respect to $B$ and consider the alternating bilinear form $\overline{\beta}\colon \overline{V} \times \overline{V}\rightarrow \F_q$
represented (with respect to $\overline{B}$) by the matrix $\overline{M}$ obtained form $M$ by removing its last row and its last column:
\[ M:=\begin{pmatrix}
    0 & 1 & \ldots & 0 & 0 & 0\\
    1 & 0 & \ldots & 0 & 0 &0\\
    \vdots&& \ddots && & \vdots\\
    0 & 0 & \ldots & 0 & 1 &0\\
    0 & 0 & \ldots & 1 & 0 &0 \\
    0 & 0& \ldots & 0& 0&0
    \end{pmatrix}\qquad {\rm and }\qquad
   \overline{M}:=\begin{pmatrix}
    0 & 1 & \ldots & 0 & 0 \\
    1 & 0 & \ldots & 0 & 0 \\
    \vdots&& \ddots && \vdots \\
    0 & 0 & \ldots & 0 & 1 \\
    0 & 0 & \ldots & 1 & 0
    \end{pmatrix}. \]
Note that $\oM$ is non-singular and $\oM^{-1}=\oM.$
Under these assumptions
$\Pi_{\varphi}=N^{\perp_{\varphi}}$ has equation
 \begin{equation}\label{eq hyperplane}
 \Pi_{\varphi}\colon \sum_{i=1}^{2n} s_{i,2n+1} x_i=0.
\end{equation}

\begin{claim}\label{claim}
The following properties hold:
\begin{enumerate}[{\rm a)}]
\item\label{Ca} $\varphi(x,y)=\overline{\varphi}(\overline{x},\overline{y}),\,\forall x,y\in \Pi_{\varphi}$;
\item\label{Cb} $\beta(x,y)=\overline{\beta}(\overline{x},\overline{y}),\,\forall x,y\in \Pi_{\varphi}.$
\end{enumerate}
\end{claim}
\begin{proof}
We shall only prove Case~\ref{Ca}), as Case~\ref{Cb}) is entirely analogous.

Let  $x=(x_i)_{i=1}^{2n+1}$ and $y=(y_i)_{1}^{2n+1}$ be the coordinates of two vectors in $\Pi_{\varphi}.$
Then, by Equation~\eqref{eq hyperplane},
\[ \sum_{i=1}^{2n}s_{i,2n+1}y_i=0\qquad {\rm and }\qquad \sum_{i=1}^{2n}s_{i,2n}x_{i}=0. \]
So we have
\begin{equation*}
\begin{array}{ll}
\varphi(x,y)&=(x_1,x_2,\dots, x_{2n+1})S\begin{pmatrix}
y_1\\
y_2\\
\vdots\\
y_{2n+1}
\end{pmatrix}
=\displaystyle{\sum_{{\begin{subarray}{c}
i,j=1\\
i<j
\end{subarray}}}^{2n+1}}s_{ij}x_jy_i + \sum_{{\begin{subarray}{l}
i,j=1\\
i<j
\end{subarray} }}^{2n+1}s_{ij}x_iy_j= \\
 &=\displaystyle{\sum_{{\begin{subarray}{l}
i,j=1\\
i<j
\end{subarray}} }^{2n}s_{ij}x_jy_i} + \sum_{{\begin{subarray}{l}
i,j=1\\
i<j
\end{subarray} }}^{2n}s_{ij}x_iy_j + x_{2n+1}\displaystyle{\sum_{i=1}^{2n}s_{i,2n+1}y_i+
y_{2n+1}\sum_{i=1}^{2n}s_{i,2n}x_{i}}=\\
 & =(x_1,x_2,\dots, x_{2n})\overline{S}\begin{pmatrix}
y_1\\
y_2\\
\vdots\\
y_{2n}
\end{pmatrix}+ x_{2n+1}\displaystyle{\sum_{i=1}^{2n}s_{i,2n+1}y_i+
y_{2n+1}\sum_{i=1}^{2n} s_{i,2n}x_{i}}=\\
 & =\overline{x}^T\overline{S}\overline{y}=\overline{\varphi}(\ox,\oy).
\end{array}
\end{equation*}
\end{proof}

The condition $u^{\perp_{\cQ}}=u^{\perp_{\varphi}}$ with $u\notin \Rad$ holds if and only if the systems of equations $x^tSu$ and $x^tMu$, where $x=(x_i)_{i=1}^{2n+1}$ are equivalent.
This means that there exists an element $\lambda\in \F_q\setminus \{0\}$ such that  $Su=\lambda Mu$.
Note that for $\lambda=0$ we have vectors $u$ in  $\Rad$ and the inclusion $u^{\perp_\cQ}\subseteq u^{\perp_{\varphi}}$ is proper.
The set
\begin{equation}\label{def U}
U:=\{ u\in\Pi_{\varphi}\cap\cQ\colon \exists\lambda\not= 0 \text{ such that } v^TSu=\lambda v^TMu,\quad  \forall v\in V\}
\end{equation}
is clearly a proper subset of
\begin{equation}\label{def U bar}
\widetilde{U}:=\{ u\in\Pi_{\varphi}\cap\cQ\colon \exists\lambda\not= 0 \text{ such that } v^TSu=\lambda v^TMu,\quad  \forall v\in \Pi_{\varphi}\}.
\end{equation}
By Claim~\ref{claim}, we have
\begin{equation}\label{U tilde}
\begin{array}{ll}
\oU&=\{ u\in\Pi_{\varphi}\cap\cQ\colon \exists\lambda\not= 0 \text{ such that } \ov^T\oS\ou=\lambda \overline{v}^T\oM\ou,\quad  \forall v\in \Pi_{\varphi}\}=\\
 & =\{ u\in\Pi_{\varphi}\cap\cQ\colon \exists\lambda\not= 0 \text{ such that } \oS\ou=\lambda \oM\ou\}=\\
 & =\{ u\in\Pi_{\varphi}\cap\cQ\colon \ou \text{ is an eigenvector of non-zero eigenvalue} \text{ for } \oM\oS\}.
\end{array}
\end{equation}
\begin{claim}\label{claim 2}
The number of eigenvectors of $\oM\oS$ of non-zero  eigenvalue is at most $q^{2r-2}.$
\end{claim}
\begin{proof}
Let $V_0:=\ker(\oM\oS)$ be the eigenspace of eigenvalue $0$ of $\oM\oS$.

To prove Claim~\ref{claim 2} we shall first show that $\dim(V_0)=\dim(\Rad)+1$.
As $\oM$ is non-singular, $V_0=\ker (\oS)$. Furthermore, since
$\oS$ is a $(2n\times 2n)$-minor of $S$, we have $\rank(S)-2\leq\rank(\oS)\leq\rank(S)$.
In particular,
$\dim(\Rad)-1\leq\dim V_0\leq\dim(\Rad)+1.$
Define $\oRad:=\{ \ox\colon x\in\Rad\}$.
We claim that $\oRad$ is a proper subspace of $V_0.$ Indeed, let $\overline{w}\in\oRad$. As $\Rad\subset \Pi_{\varphi}$ and $w\in \Rad$
we have, by Claim~\ref{claim}, that $\overline{\varphi}(\overline{w}, \ox)=0= \varphi(w,x)$ for any $x\in \Pi_{\varphi}$. This implies  $\oS\overline{w}=0$; so $\overline{w}\in V_0.$
Furthermore, $\dim(\oRad)=\dim(\Rad)$. Indeed, let $(b_1,\dots, b_{2(n-r)+1})$ be a basis of $\Rad$ then $(\overline{b}_1,\dots, \overline{b}_{2(n-r)+1})$ is clearly a generating set for $\oRad$. If the
latter vectors were to be linearly dependent, then there would be $\alpha_1,\ldots,\alpha_{2(n-r)+1}$,
not all zero, such that
$\alpha_1\overline{b}_1+\cdots+\alpha_{2(n-r)+1}\overline{b}_{2(n-r)+1}=0.$
Then $v:=\alpha_1{b}_1+\cdots+\alpha_{2(n-r)+1}{b}_{2(n-r)+1}\neq 0$ and $\overline{v}=0$.
This means $v=\gamma e_{2n+1}\in\Rad$ for some $\gamma\neq0$, a contradiction as $N\not\subseteq\Rad$.
So,  $\dim(V_0)\geq\dim(\oRad)=\dim(\Rad)=2(n-r)+1$.
Since $S$ is an antisymmetric matrix of odd order, $\dim(\oRad)=\dim(\Rad)=\dim(\ker(S))$ is odd.
On the other hand $\oS$ is, by construction, an antisymmetric matrix of
even order, so $\dim(V_0)=\dim(\ker(\oS))$ is even; hence $\dim(V_0)\neq\dim(\Rad)=2(n-r)+1$.
It follows that
\begin{equation}\label{dim ker}
\dim(V_0)=2n-2r+2=\dim(\Rad)+1.
\end{equation}

Suppose now that there are $t\geq 0$ eigenspaces $V_{\lambda_1},\ldots,V_{\lambda_t}$ of $\oM\oS$ of non-zero eigenvalues $\lambda_1,\ldots,\lambda_t$  and let $d_i:=\dim(V_{\lambda_i})\geq 1.$
 Note that if $t=0$ then we immediately have $A>0.$
Suppose also $d_1\leq d_2\leq\cdots\leq d_t$.
Then,
\[\sum_{i=1}^td_i +\dim(V_0)\leq \dim(\overline{V}).\]
By Equation~\eqref{dim ker}, we have
\begin{equation}
\label{maxdim MS}
\sum_{i=1}^td_i \leq 2n -\dim(V_0)\leq 2r-2;
\end{equation}
so,
by the properties of the exponential function,
\[\sum_{i=1}^t (|V_{\lambda_i}|-1)=\sum_{i=1}^t(q^{d_i}-1)\leq q^{\sum_{i=1}^td_i}-t\leq q^{2r-2}.\]
\end{proof}

\par
\noindent\fbox{Suppose $3\leq\dim(\Rad)\leq 2n-3.$} By Equation~\eqref{def A},
\[A=q^{2n-2}-1-|\{\Rad\cap \cQ\}|-|\{u\in\cQ\colon u^{\perp_{\cQ}}=u^{\perp_{\varphi}}\}|;\]
using Equations~\eqref{def A-1},~\eqref{def U},~\eqref{def U bar} and~\eqref{U tilde},
\[A\geq q^{2n-2}-1-(|\Rad|-1) -|U|\geq q^{2n-2}-q^{2n-2r+1}-|\widetilde{U}|.\]
By Claim~\ref{claim 2}, $|\oU|\leq q^{2r-2}$; hence,
\begin{equation}\label{A caso 1}
A\geq q^{2n-2} -q^{2n-2r+1} - q^{2r-2}.
\end{equation}
Under the assumption
$3\leq\dim(\Rad)\leq 2n-3$, we have $2\leq r\leq n-1.$ So, $2r-2\leq 2n-4$ and $2n+1-2r\leq 2n-3$.
By taking these two inequalities into account in Equation~\eqref{A caso 1} we get
\[A\geq q^{2n-2} -q^{2n-2r+1} - q^{2r-2} \] 
Observe that the function $f(r):=q^{2n-2r+1}+q^{2r-2}$, regarded as defined over the reals,
has derivative $\frac{\partial f}{\partial r}=2\log(q)\left(q^{2r-2}-q^{2n-2r+1}\right)$.
In particular $f(r)$ is decreasing for $2\leq r<\frac{2n+3}{2}$.
So it attains its maximum for $r=2$ and
\[ A\geq q^{2n-2}-q^{2n-3}-q^2>0. \]
Note that if $r=n$ the last inequality does not hold.
This completes the proof for $3\leq\dim(\Rad)\leq 2n-3$.
\medskip
\par
\noindent\fbox{Suppose $\dim(\Rad)=1.$} This is equivalent to say $r=n$. In this case, by Equation~\eqref{dim ker},
$\dim V_0=\dim(\ker(\oS))=2$.
By Equation~\eqref{maxdim MS}, the maximum dimension of an eigenspace of $\oM\oS$ is $2n-2$.


Define \[\oPi:=\{\ox\colon x\in \cQ\cap \Pi_{\varphi}\}.\]
Then, $\oPi$ is the hyperplane of $\oV$ of equation $\sum_{i=1}^{2n}s_{i,2n+1}x_i=0$ and the map $\xi:x\to\ox$ is
a bijection between the points of $\cQ\cap\Pi_{\varphi}$ and those of $\oPi$.
\begin{claim}\label{claim 3}
\[|\oU|=\sum_{\lambda\not=0}|\oPi \cap V_{\lambda}|.\]
\end{claim}
\begin{proof}
For $\lambda$ a non-zero eigenvalue of $\oM\oS$, define $\oU_{\lambda}:=\{ u\in\Pi_{\varphi}\cap\cQ\colon \oM\oS\ou=\lambda \ou\}.$ By the above considerations, $|\oU_{\lambda}|=|\oPi\cap V_{\lambda}|$.
Furthermore, $\oU_{\lambda}\cap\oU_{\mu}=\emptyset$ for $\lambda\neq\mu$ and, by~\eqref{U tilde},
$\oU=\bigcup_{\lambda\neq0}\oU_{\lambda}$.
This proves the claim.
\end{proof}

\begin{claim}\label{claim 4}
Assume that there are $t>0$ distinct eigenspaces $V_{\lambda_i}$ for $\oM\oS$ of non-zero eigenvalue.
Then,
\[ \sum_{i=1}^t |\oPi\cap V_{\lambda_i}|\leq |\oPi\cap Z|, \]
where $Z=\oplus_{i=1}^tV_{\lambda_i}$ and $\dim(Z)\leq 2n-2.$
\end{claim}
\begin{proof}
Suppose $t\geq 2$.
Take two eigenspaces $V_{\lambda_1}$ and $V_{\lambda_2}$ of $\oM\oS$ with dimension respectively  $d_1,d_2\leq 2n-2$ and define $Z:=V_{\lambda_1}\oplus V_{\lambda_2}$.
As $V_{\lambda_1}\cup V_{\lambda_2}\leq Z$, we have $|\oPi\cap V_{\lambda_1}|+|\oPi\cap V_{\lambda_2}|=
|\oPi\cap(V_{\lambda_1}\cup V_{\lambda_2})|\leq|\oPi\cap Z|$.
So
\[ \sum_{i=1}^t |V_{\lambda_i}\cap \oPi|\leq  |Z\cap\oPi|+\sum_{i=3}^t |V_{\lambda_i}\cap \oPi|. \]
Iterating this procedure $t-1$ times we get
$\sum_{i=1}^t |V_{\lambda_i}\cap\oPi| \leq |Z'\cap\oPi|$
where $Z':=\oplus_{i=1}^tV_{\lambda_i}$.
As $\sum d_i\leq 2n-2$, we have the claim.
\end{proof}
Using Claim~\ref{claim 4}, we see that a matrix $\oM\oS$ having the maximum number of eigenvectors
(of non-null eigenvalues) can be taken so that it admits
exactly one eigenspace $V_{\lambda}$ with
$\dim V_{\lambda}=2n-2$. We shall assume this to be the case in the remainder of the section.
So, by Equation~\eqref{U tilde} and Claim~\ref{claim 3},
\begin{equation}\label{A last}
A\geq q^{2n-2}-1-(q-1)- \sum_{i=1}^t|\oPi \cap V_{\lambda_i}|\geq q^{2n-2}-q- |V_{\lambda}\cap\oPi|.
\end{equation}

Recall that $\oPi$ is a hyperplane of $\oV$; so $|V_{\lambda}\cap\oPi|$ can assume only two values depending on whether  $\oPi$ intersects $V_{\lambda}$ in a hyperplane or $\oPi$ properly contains $V_{\lambda}$.
In the former case, $\dim(V_{\lambda}\cap\oPi)=2n-3$; hence $|V_{\lambda}\cap\oPi|=q^{2n-3}$ and Equation~\eqref{A last} gives $A>0$, proving the lemma.

In the latter case, Equation~\eqref{A last} is not sufficient, as it gives $A\geq -q$.
To rule out this possibility we need a more accurate lower bound for $A$. To this aim, consider Equation~\eqref{def A} under the assumption $V_{\lambda}\subseteq\oPi$. We have
\[A\geq q^{2n-2}-1-(q-1)-|\{u\in \cQ\cap \Pi_{\varphi}\colon u^{\perp_{\cQ}}=u^{\perp_{\varphi}}\}|.\]
Also,
\begin{equation}\label{A fin}
\begin{array}{l}
|\{u\in \cQ\cap \Pi_{\varphi}\colon u^{\perp_{\cQ}}=u^{\perp_{\varphi}}\}|
=|\{ u\in\Pi_{\varphi}\cap\cQ\colon x^tSu=\lambda x^tMu, \forall x\in V \}|=\\
= |\{u\in \cQ\cap \Pi_{\varphi}\colon x^t(S-\lambda M)u=0,\ \, \forall x\in V\}|
\leq |\{u\in \Pi_{\varphi}\colon u\in \ker(S-\lambda M)\}|.
\end{array}
\end{equation}

Observe that to any vector in $\bar{y}\in\ker (\oS-\lambda \oM)$ there correspond at most one vector
$y\in\ker (S-\lambda M)$ because if $y_1,\, y_2\in\ker(S-\lambda M)$, $y_1\not=y_2$ and $\oy_1=\oy_2\in\ker(\oS-\lambda \oM)$, then $\langle y_1-y_2\rangle=\langle e_{2n+1}\rangle=N$ and
$e_{2n+1}\in\ker (S-\lambda M)$. As $Me_{2n+1}=0$, this implies
$x^T Se_{2n+1}=0,\,\forall x\in V$; hence, $\langle e_{2n+1}\rangle =N\subseteq\Rad$, against our hypothesis.

So,
$\dim(\ker(\oS-\lambda\oM))-1\leq\dim(\ker (S-\lambda M) )\leq \dim(\ker(\oS-\lambda \oM)).$
By construction, $\dim(\ker(\oS-\lambda \oM))=\dim(V_{\lambda})\leq 2n-2$; furthermore, $\dim(\ker (S-\lambda M) )$ is odd because $S-\lambda M$ is a $(2n+1)\times (2n+1)$-antisymmetric matrix. So,
$\dim(\ker (S-\lambda M))=2n-3$.

As $(n,q)\neq(2,2)$,
by Equations~\eqref{def A} and~\eqref{A fin}  we have
\[A\geq q^{2n-2}-1-(q-1)-|\{u\in V\colon u\in \ker(S-\lambda M)\}|\geq q^{2n-2}-q^{2n-3}-q>0.\]
This proves the lemma.
\end{proof}

Combining Equation~\eqref{eq peso final},  Lemma~\ref{lemma N in R} and Lemma~\ref{lemma N not in R}, we have
the following.
\begin{corollary}\label{N not in R co}
  If $N\subseteq \Rad$ or $N\not\subseteq\Rad$ and $\dim(\Rad)<2n-1$, then $\wt(\varphi)>q^{4n-5}-q^{3n-4}$.
\end{corollary}

There remains to consider the class
of alternating bilinear forms having radical of maximum dimension not containing the nucleus $N.$

\begin{lemma}\label{min weight}
If $N\not\subseteq \Rad$ and $\dim(\Rad)=2n-1$, then $\wt(\varphi)\geq q^{4n-5}-q^{3n-4}.$
If $\dim \Rad=2n-1$ and $\Rad\cap\cQ$ is a cone of vertex a point $P$ projecting a hyperbolic quadric $Q^+(2n-3,q)$, then $\wt(\varphi)= q^{4n-5}-q^{3n-4}.$
\end{lemma}
\begin{proof}
As $\dim\Rad=2n-1$, a line $\ell$ of $\cQ$ is totally isotropic for $\varphi$ if and only if $\ell\cap\Rad\neq\{0\}$.
To determine the weight $\wt(\varphi)$ of $\varphi$ we just
need to determine the number of totally singular lines of $\cQ$ with non-trivial intersection with $\Rad$.

Let $P\in\cQ\cap\Rad$, then \emph{all} lines through $P$ meet $\Rad$ non-trivially. There are
  exactly $(q^{2n-2}-1)/(q-1)$ such lines. Each line $\ell$ contained in $\cQ\cap \Rad$ ends up being counted $(q+1)$ times; so we need to determine the
number
\[ |(\cQ\cap \Rad)|\frac{q^{2n-2}-1}{q-1}-q|\{\mbox{totally singular lines contained in $\Rad$}\}|. \]
Denote the number of totally singular lines contained in $\Rad$ by $\sigma(\cQ\cap \Rad)$.

There are four types of sections obtained by intersecting a parabolic quadric $\cQ$ with a space $\Pi_a\cap\Pi_b$ of codimension $2$; indeed
\begin{enumerate}
  \item if $\Pi_a\cap\cQ$ is an elliptic quadric $Q^-(2n-1,q)$, then $\Pi_a\cap \Pi_b\cap\cQ$ is either a parabolic quadric $Q(2n-2,q)$ or
    a cone over an elliptic quadric $Q^-(2n-3,q)$.
  \item if $\Pi_a\cap\cQ$ is a hyperbolic quadric $Q^+(2n-1,q)$, then $\Pi_a\cap \Pi_b\cap\cQ$ is either a parabolic quadric $Q(2n-2,q)$ or
    a cone over a hyperbolic quadric $Q^+(2n-3,q)$.
  \item if $\Pi_a$ is tangent to $\cQ$, then $\Pi_a\cap \cQ$ is a cone over a parabolic quadric $Q(2n-2,q)$.
      The possible intersections of $\Pi_a\cap \cQ$ with $\Pi_b$ are now:
    \begin{enumerate}
      \item a parabolic quadric $Q(2n-2,q)$ (if $\Pi_b$ does not pass through the vertex of $\Pi_a\cap\cQ$);
      \item a quadric with vertex a line and basis a parabolic quadric $Q(2n-4,q)$;
      \item a cone over a hyperbolic quadric $Q^+(2n-3,q)$;
      \item a cone over an elliptic quadric $Q^-(2n-3,q)$.
\end{enumerate}
\end{enumerate}
So $\cQ\cap\Rad$ is either:
  \begin{enumerate}[a)]
    \item a parabolic quadric $Q(2n-2,q)$; then,
  \[ |\cQ\cap \Rad|=\frac{q^{2n-2}-1}{q-1},\qquad \sigma(\cQ\cap \Rad)=\frac{(q^{2n-4}-1)(q^{2n-2}-1)}{(q^2-1)(q-1)}; \]
  then the weight is
  \[ \wt(\varphi)=q^{4n-5}-q^{2n-3}; \]
  \item a cone of vertex a point $P$ over a hyperbolic quadric $Q^+(2n-3,q)$; then,
  \[ |\cQ\cap \Rad|=q\frac{(q^{n-1}-1)(q^{n-2}+1)}{q-1}+1; \]
  \[ \sigma(\cQ\cap \Rad)=\frac{(q^{n-1}-1)(q^{n-2}+1)}{q-1}+q^2
    \frac{(q^{2n-4}-1)(q^{n-1}-1)(q^{n-3}+1)}{(q^2-1)(q-1)}; \]
  then the weight is
  \[ \wt(\varphi)=q^{4n-5}-q^{3n-4}; \]
  \item a cone of vertex a point $P$ over an elliptic quadric $Q^-(2n-3,q)$; then,
    \[ |\cQ\cap \Rad|=q\frac{(q^{n-1}+1)(q^{n-2}-1)}{q-1}+1; \]
    \[
      \sigma(\cQ\cap \Rad)=
\frac{(q^{n-1}+1)(q^{n-2}-1)}{q-1}+q^2\frac{(q^{2n-4}-1)(q^{n-1}+1)(q^{n-3}-1)}{(q^2-1)(q-1)}; \]
then the weight is
\[ \wt(\varphi)=q^{4n-5}+q^{3n-4}; \]
\item a singular quadric with vertex a line $\ell$ and basis a
  parabolic quadric $Q(2n-4,q)$; then
  \[ |\cQ\cap\Rad|=q^2\frac{q^{2n-4}-1}{q+1}+(q+1)=\frac{q^{2n-2}-1}{q-1}; \]
  \[ \sigma(\cQ\cap \Rad)=    1+q\frac{q^{2n-4}-1}{q-1}+q^2\left(q^2\frac{(q^{2n-6}-1)(q^{2n-4}-1)}{(q^2-1)(q-1)}+\frac{(q^{2n-4}-1)}{q-1}\right);
 \]
then the weight is
\[ \wt(\varphi)=q^{4n-5}. \]
\end{enumerate}

\end{proof}

\begin{theorem}\label{thm N not in R}
If $N\not\subseteq \Rad$ then $\wt(\varphi)\geq q^{4n-5}-q^{3n-4}.$
Moreover, if $N\not\subseteq \Rad$ and $\dim(\Rad)=2n-1$ there exist codewords of weight $q^{4n-5}-q^{3n-4}.$
\end{theorem}
\begin{proof}
The theorem follows from Corollary~\ref{N not in R co} and Lemma~\ref{min weight}.
\end{proof}
Combining the result for $(q,n)=(2,2)$,
Theorem~\ref{thm N not in R} and  Lemma~\ref{min weight}, we have the following.
\begin{corollary}\label{corollario proj equiv}
All minimum weight codewords of $\cP_{n,2}$ for $q$ even are projectively equivalent.
\end{corollary}

\begin{corollary}
  \label{second distance}
For $n\neq 3$, the second smallest distance of the code is $q^{4n-5}-q^{2n-3}$.
\end{corollary}
\begin{proof}
  For $n=2$, using the argument of \cite[\S 2.3.2]{IL0}, the full
  spectrum of the code can be determined, and the second smallest
  weight is $q^3-q$.
  For $n=3$
  For $n\geq 4$,
  by Lemma~\ref{lemma N not in R},
  when $N\not\subseteq\Rad$ 
  we have $A\geq q^{2n-2}-q^{2n-3}-q^2$.
  Plugging this in~\eqref{eq peso final}, we immediately obtain
  $\wt(\varphi)>q^{4n-5}-q^{2n-3}$ for $n\geq 4$.
\end{proof}
  Our Main Theorem  follows from Corollary~\ref{N not in R co}, Theorem~\ref{thm N not in R}, Corollary~\ref{corollario proj equiv}
  and Corollary~\ref{second distance}.
  \par\noindent
\ ~\hfill$\square$
\begin{remark}
  Using the estimates appearing in the proof of \cite[Theorem 3.7]{ILP}
  in \cite[Equation~(9)]{ILP}, a straightforward computation
  shows that
  the
  second smallest distance of a line polar Grassmann code of orthogonal
  type is $q^{4n-5}-q^{2n-3}$ for $n>3$ also when $q$ is odd.

  We conjecture that also for $n=3$ the
  second smallest distance of the code is 
  $q^7-q^3$ (both for $q$ even and $q$ odd).
\end{remark}
\section*{Acknowledgements}
Both authors are affiliated with GNSAGA of INdAM (Italy) whose support they
acknowledge.
We thank the anonymous referee of the paper for having pointed out the
value of the second smallest distance of the code.

\vskip.2cm\noindent
\begin{minipage}[t]{\textwidth}
Authors' addresses:
\vskip.2cm\noindent\nobreak
\centerline{
\begin{minipage}[t]{7cm}
Ilaria Cardinali\\
Department of Information Engineering and Mathematics\\University of Siena\\
Via Roma 56, I-53100, Siena, Italy\\
ilaria.cardinali@unisi.it\\
\end{minipage}\hfill
\begin{minipage}[t]{6cm}
Luca Giuzzi\\
D.I.C.A.T.A.M. \\ Section of Mathematics \\
Universit\`a di Brescia\\
Via Branze 43, I-25123, Brescia, Italy \\
luca.giuzzi@unibs.it
\end{minipage}}
\end{minipage}
\end{document}